\documentclass{article}
\usepackage[latin1]{inputenc}
\usepackage[small]{titlesec}
\usepackage{amsmath}
\usepackage{amsthm}
\usepackage{comment}
\usepackage{amssymb}
\usepackage{graphicx}
\usepackage{standalone}
\usepackage{color}
\theoremstyle{plain}
\newtheorem{thm}{Theorem}[section]
\newtheorem{lem}[thm]{Lemma}
\newtheorem{prop}[thm]{Proposition}
\newtheorem{cor}[thm]{Corollary}

\newtheorem{rem}[thm]{Remark}
\newtheorem{prob}[thm]{Problem}
\newtheorem{defn}[thm]{Definition}

\usepackage[small]{titlesec}
\usepackage{blindtext}

	\title{Classification of Deza graphs from anisotropic association schemes of quadrics}
\author{Valentino Smaldore \footnote{Valentino Smaldore:
valentino.smaldore@unipd.it 
Dipartimento di Tecmica e Gestione dei Sistemi Industriali,
Universit\`{a} degli Studi di Padova, Stradella San Nicola 2, 36100 Vicenza, Italy.}}

\date{}

\begin{document}
\maketitle
\begin{abstract}
Let $Q^\varepsilon(3,q)$, where
$\varepsilon\in\{+,-\}$ and $q>3$ is odd, be a non-degenerate hyperbolic or elliptic quadric of $PG(3,q)$. Fix one of the two quadratic classes of anisotropic points. Since the line joining two distinct points of this class is tangent, secant, or external to the quadric, one obtains a $3$-class association scheme. We classify all non-trivial unions of its relations which define Deza graphs. In addition to the previously known tangency family, exactly four exceptional strictly Deza graphs occur, with parameters $(360,135,54,45)$, $(369,108,36,27)$, $(65,34,18,15)$ and $(168,111,75,70)$. We determine their spectra and Deza children and give geometric or group-theoretic descriptions of all four exceptional graphs. 
\end{abstract}

\section{Introduction}

Strongly regular graphs arising from finite classical polar spaces form one of the standard bridge between finite geometry and algebraic combinatorics. Natural weakening of strong regularity is given by Deza
graphs. A $k$-regular graph on $v$ vertices is a Deza graph with parameters $(v,k,b,a)$ if the number of common neighbours of two distinct vertices assumes exactly two values $a$ and $b$, where $b\geq a$, see \cite{EricksonEtAl1999} for more information. If the graph has diameter two and is not strongly regular, it is usually called \emph{strictly Deza}. 
Association schemes provide a natural source of such graphs. If $\mathfrak X=(X,\{R_0,R_1,\ldots,R_d\})$ is a symmetric $d$-class association scheme with intersection numbers $p_{ij}^h$, and if $S\subseteq\{1,\ldots,d\}$, then the graph whose adjacency matrix
is $A_S=\sum_{i\in S}A_i$ has $\sum_{i,j\in S}p_{ij}^{h}$ common neighbours for any pair of vertices belonging to the relation $R_h$. Thus a union of relations is a Deza graph precisely when these quantities assume at most two values as $h$ varies. This point of view has already proved useful in the study of Deza graphs arising from $3$-class schemes; see, for instance, \cite{GoryainovShalaginov2021}. Here, we consider the association scheme arising from anisotropic points of the quadrics $Q^\varepsilon(3,q)\subseteq PG(3,q)$, $\varepsilon\in\{+,-1\}$, where $\varepsilon=+$ denotes the hyperbolic case and
$\varepsilon=-1$ the elliptic case. We fix one quadratic class of anisotropic points. Two such points determine a line which is tangent, secant, or external to the quadric, giving three non-trivial relations. The association schemes on anisotropic points of quadrics had been intensively studied. In particular, character tables for schemes arising from
finite orthogonal groups were obtained by Bannai, Hao and Song \cite{BannaiHaoSong1990}.  
The dimension-three elliptic case is closely
connected with inversive planes; see \cite{FisherEtAl1989}. A recent uniform treatment of the schemes on anisotropic points of finite quadrics was given by Adriaensen and De Boeck \cite{AdriaensenDeBoeck2025}.

The graph corresponding to the tangent relation is already known to be a Deza graph. In the language of finite circle geometries, Adriaensen proved that the relevant component of the $1$-intersecting graph has
exactly two possible numbers of common neighbours
\cite[Lemma~4.6]{Adriaensen2022}. More recently, Cuypers considered graphs on non-isotropic points in which two vertices are adjacent when they span a tangent line, calling them \emph{orthogonal graphs}
\cite{Cuypers2026}.

Our main purpose is to classify all Deza graphs obtained as non-trivial unions of relations of this $3$-class anisotropic association scheme. The tangency relation is already known to yield a Deza graph, and is therefore regarded here as the motivating example. The new contribution is the complete classification of all such unions. Besides the known tangency family, exactly four exceptional strictly Deza graphs occur. We further determine their spectra and Deza children and give geometric or group-theoretic descriptions of all four exceptional cases.

\section{Preliminaries}

\subsection{Deza graphs}

\begin{defn}
A graph $\Gamma$ is a \emph{Deza graph} with parameters $(v,k,b,a)$ if it has $v$ vertices, is $k$-regular, and every two distinct vertices have either $a$ or $b$ common neighbours, where $b\geq a$.
\end{defn}

A strongly regular graph is therefore a Deza graph, but in a strongly regular graph the number of common neighbours is prescribed by whether the two vertices are adjacent. This need not be true for a general
Deza graph.

\begin{defn}
A Deza graph is called \emph{strictly Deza} if it has diameter two and is not strongly regular.
\end{defn}

\begin{defn}
Let $\Gamma$ be a Deza graph with parameters $(v,k,b,a)$, where $b\geq a$. For distinct vertices $x,y\in V(\Gamma)$, define two graphs $\Gamma_a$ and $\Gamma_b$ on the same vertex set by
\[
x\sim_{\Gamma_a} y
\quad\Longleftrightarrow\quad
|N_\Gamma(x)\cap N_\Gamma(y)|=a,
\]
and
\[
x\sim_{\Gamma_b} y
\quad\Longleftrightarrow\quad
|N_\Gamma(x)\cap N_\Gamma(y)|=b.
\]
The graphs $\Gamma_a$ and $\Gamma_b$ are called the
\emph{children} of $\Gamma$.
\end{defn}

\begin{rem}
The two children are complementary graphs, since every pair of distinct vertices of $\Gamma$ has either $a$ or $b$ common neighbours.
\end{rem}

\subsection{Association schemes}

Given a set $X$, the concept of an \textit{association scheme} is introduced to represent certain relations between pairs of elements in $X$. From now on, the set $X$ is assumed to be finite.
\begin{defn}
   For a set $X$ and the Cartesian product
   $$X\times X=\{(\alpha,\beta)|\alpha,\beta\in X\},$$
   an association scheme with $d$ associate classes on $X$ is a partition of $X\times X$ into $d+1$ associate classes (also called relations) $R_{0}, R_{1},\ldots, R_{d}$ such that:
   \begin{enumerate}
    \item $R_{0}=Diag(X)=\{(\alpha,\alpha)|\alpha\in X\}$;
    \item for all $i$ in $\{0,1,\ldots, d\}$, $R_{i}$ is symmetric, i.e. $(\alpha,\beta)\in R_{i}$ if and only if $(\beta,\alpha)\in R_{i}$;
    \item for all $i,j,k$ in $\{0,1,\ldots,d\}$ there exists an integer $p_{ij}^{k}$ such that, for all $(\alpha,\beta)\in R_{k}$:
         $$|\{\gamma\in X|(\alpha,\gamma)\in R_{i} \wedge (\gamma,\beta)\in R_{j}\}|=p_{ij}^{k}.$$
   \end{enumerate}
  \end{defn}
  The symmetry condition also says $p_{ij}^{0}=0$ for $i\neq j$. Similarly $p_{i0}^{k}=0$ for $i\neq k$ and $p_{0j}^{k}=0$ for $j\neq k$, while $p_{i0}^{i}=1=p_{0i}^{i}$. The number $p_{ii}^{0}=a_{i}$ is called the \textit{valency} of the $i$-th associate class.

\section{The anisotropic scheme of a quadric}

Let $q>3$ be an odd prime power, and let $Q=Q^\varepsilon(3,q)$ be a non-degenerate quadric of $PG(3,q)$, where $\varepsilon\in\{+,-\}$.  Let $\beta$ be a quadratic form defining
$Q$. The anisotropic points are split into two quadratic classes according to the square class of $\beta(X)$. Fix one of these classes and denote it by $\mathcal P$.

For two distinct points $X,Y\in\mathcal P$, define
\[
\begin{aligned}
 (X,Y)\in R_1
 &\quad\Longleftrightarrow\quad
 |\langle X,Y\rangle\cap Q|=1,\\
 (X,Y)\in R_2
 &\quad\Longleftrightarrow\quad
 |\langle X,Y\rangle\cap Q|=2,\\
 (X,Y)\in R_3
 &\quad\Longleftrightarrow\quad
 |\langle X,Y\rangle\cap Q|=0.
\end{aligned}
\]
Thus, $R_1$, $R_2$, and $R_3$ correspond, respectively, to tangent, secant, and external lines. Together with the diagonal relation $R_0$, these relations form a symmetric $3$-class association scheme on $\mathcal{P}$; see \cite{AdriaensenDeBoeck2025}. The number of vertices is $v=|\mathcal{P}|=\frac{q(q^2-\varepsilon)}{2}$. Let $A_i$ be the adjacency matrix of $R_i$. With the present labeling, the valencies are as follows. \begin{equation}\label{eq:valencies}
 k_1=q^2-1,\qquad
 k_2=\frac14q(q+\varepsilon)(q-3),\qquad
 k_3=\frac14q(q-\varepsilon)(q-1).
\end{equation}

For every non-empty subset $S\subseteq\{1,2,3\}$, define the graph
\[
 \Gamma_S=(\mathcal{P},E_S),
 \qquad
 E_S=\bigcup_{i\in S}R_i,
\]
with adjacency matrix $A_S=\sum_{i\in S}A_i$.

We shall use the following standard criterion for Deza graphs arising
from association schemes; see \cite[Theorem~4.2]{EricksonEtAl1999}.

\begin{lem}\label{lem:deza-criterion}
For $h\in\{1,2,3\}$, set $c_h(S)=\sum_{i,j\in S}p_{ij}^h$. If $(X,Y)\in R_h$, then $X$ and $Y$ have exactly $c_h(S)$ common neighbours in $\Gamma_S$. Consequently, $\Gamma_S$ is a Deza graph if and only if $|\{c_1(S),c_2(S),c_3(S)\}|\leq2$.
\end{lem}

\begin{proof}
The $(X,Y)$-entry of $A_iA_j$ is the number of vertices $Z$ such that $(X,Z)\in R_i$ and $(Z,Y)\in R_j$.  If $(X,Y)\in R_h$, this number is
$p_{ij}^{h}$. Hence \[
 (A_S^2)_{XY}
 =\sum_{i,j\in S}(A_iA_j)_{XY}
 =\sum_{i,j\in S}p_{ij}^{h},
\]
which proves the claim.
\end{proof}

\section{The tangency graph}
We first recall the known infinite family arising from the tangent relation.

\begin{prop}\label{prop:tangent}
Let $S=\{1\}$. Then $\Gamma_{\{1\}}$ is a Deza graph with parameters
\[
 \left(
 \frac{q(q^2-\varepsilon)}2,\,
 q^2-1,\,
 2(q+1),\,
 2(q-1)
 \right).
\]
More precisely,
\[
 p_{11}^{1}=2(q-1),\qquad
 p_{11}^{2}=2(q-\varepsilon),\qquad
 p_{11}^{3}=2(q+\varepsilon).
\]
Thus, for $\varepsilon=+$,
\[
 c_1=c_2=2(q-1),\qquad c_3=2(q+1),
\]
whereas, for $\varepsilon=-$,
\[
 c_1=c_3=2(q-1),\qquad c_2=2(q+1).
\]
For $q>3$ the graph is not strongly regular.
\end{prop}

\begin{rem}
The Deza property of this graph is already present in the literature. In the circle-geometric formulation, it is explicitly stated in \cite[Lemma~4.6]{Adriaensen2022}. Hence, the tangency graph should be regarded as the motivating example rather than as a new construction.
\end{rem}

\section{Classification of the Deza unions}

We now apply Lemma~\ref{lem:deza-criterion} to every proper subset $S\subsetneq\{1,2,3\}$. The intersection numbers of the anisotropic scheme allow this condition to be solved completely.
\begin{thm}\label{thm:classification}
Let $q>3$ be an odd prime power and
$\varepsilon\in\{+,-\}$.  Let
$\varnothing\neq S\subsetneq\{1,2,3\}$.
Then $\Gamma_S$ is a Deza graph if and only if one of the following occurs:
\[
\begin{array}{c@{\qquad}c@{\qquad}c}
\hline
S & q & \varepsilon\\
\hline
\{1\}   & \text{arbitrary} & \pm\\
\{2\}   & 9                & \pm\\
\{1,2\} & 5                & -\\
\{1,3\} & 7                & +\\
\hline
\end{array}
\]
The first row is the known tangency family.  The remaining four graphs have parameters
\[
\begin{aligned}
 \Gamma_{\{2\}}\bigl(Q^+(3,9)\bigr)
   &:\quad (360,135,54,45),\\
 \Gamma_{\{2\}}\bigl(Q^-(3,9)\bigr)
   &:\quad (369,108,36,27),\\
 \Gamma_{\{1,2\}}\bigl(Q^-(3,5)\bigr)
   &:\quad (65,34,18,15),\\
 \Gamma_{\{1,3\}}\bigl(Q^+(3,7)\bigr)
   &:\quad (168,111,75,70).
\end{aligned}
\]
\end{thm}

\begin{proof}
By Lemma~\ref{lem:deza-criterion}, for each $S$ it suffices to determine when at least two of $c_1(S)$,$c_2(S)$ and $c_3(S)$ coincide. We substitute the intersection numbers of the $3$-class
subscheme obtained from \cite[Lemmas~4.12, 4.14 and 4.15]{AdriaensenDeBoeck2025}.
For $S=\{1\}$, the result follows from Proposition~\ref{prop:tangent}. For $S=\{2\}$, the three pairwise differences simplify to
\[
\begin{aligned}
 c_1-c_2&=-\varepsilon\frac{(q-9)(q+1)}8,\\
 c_1-c_3&= \varepsilon\frac{(q-3)(q+3)}8,\\
 c_2-c_3&= \varepsilon\frac{q^2-4q-9}{4}.
\end{aligned}
\]
For an odd prime power $q>3$, the only admissible zero is $q=9$ in the first equation. Hence $\Gamma_{\{2\}}$ is Deza precisely when $q=9$, for either value of $\varepsilon$.

For $S=\{3\}$, we obtain
\[
\begin{aligned}
 c_1-c_2&=-\varepsilon\frac{(q-1)(q+1)}8,\\
 c_1-c_3&= \varepsilon\frac{q^2-8q-1}{8},\\
 c_2-c_3&= \varepsilon\frac{q^2-4q-1}{4}.
\end{aligned}
\]
The first expression never vanishes for $q>3$, while the roots of the other two quadratic polynomials are non-integral. Therefore, no graph $\Gamma_{\{3\}}$ is Deza.

For $S=\{1,2\}$, the differences are
\[
\begin{aligned}
 c_1-c_2&=-\varepsilon\frac{q^2-1}{8},\\
 c_1-c_3&=\frac{\varepsilon(q^2-8q-1)-16}{8},\\
 c_2-c_3&=\frac{\varepsilon(q^2-4q-1)-8}{4}.
\end{aligned}
\]
The first expression is non-zero. If $\varepsilon=-$, then $c_1-c_3=-\frac{(q-3)(q-5)}8$,
so the unique admissible value is $q=5$. If $\varepsilon=+$, neither of the remaining equations has an admissible integral solution.

For $S=\{1,3\}$, we obtain
\[
\begin{aligned}
 c_1-c_2&=-\frac{\varepsilon(q^2-8q-9)+16}{8},\\
 c_1-c_3&= \varepsilon\frac{q^2-9}{8},\\
 c_2-c_3&= \frac{\varepsilon(q^2-4q-9)+8}{4}.
\end{aligned}
\]
When $\varepsilon=+$, the first equation becomes $c_1-c_2=-\frac{(q-1)(q-7)}8$, and yields $q=7$. The other possible zero is givn by $q=3$, which is excluded, and the elliptic case gives no admissible solution.

Finally, for $S=\{2,3\}$ the differences collapse to
\[
 c_1-c_2=2\varepsilon,\qquad
 c_1-c_3=-2\varepsilon,\qquad
 c_2-c_3=-4\varepsilon,
\]
so the three common-neighbour numbers are always distinct.

Substitution at the four exceptional values gives respectively
\[
\begin{array}{c|c}
\Gamma & (c_1,c_2,c_3)\\
\hline
\Gamma_{\{2\}}(Q^+(3,9)) & (54,54,45)\\
\Gamma_{\{2\}}(Q^-(3,9)) & (27,27,36)\\
\Gamma_{\{1,2\}}(Q^-(3,5)) & (18,15,18)\\
\Gamma_{\{1,3\}}(Q^+(3,7)) & (75,75,70).
\end{array}
\]
Together with the valencies in \eqref{eq:valencies}, these give the stated parameters. 

\end{proof}

\begin{cor}
Apart from the known tangency family, exactly four Deza graphs arise
as non-trivial unions of relations of the anisotropic $3$-class scheme.
Their parameters are
\[
(360,135,54,45),\qquad
(369,108,36,27),\qquad
(65,34,18,15),\qquad
(168,111,75,70).
\]
\end{cor}

\begin{cor}
All four exceptional graphs are strictly Deza.
\end{cor}

\begin{proof}
In each case both possible numbers of common neighbours are positive, so every two vertices are at distance at most two. Since the graphs are not complete, their diameter is two. Moreover, for the two graphs arising from $S=\{2\}$ the two intersection numbers occur among non-adjacent pairs, while for the graphs arising from $S=\{1,2\}$ and $S=\{1,3\}$ they occur among adjacent pairs. Hence none of the four graphs is strongly regular.
\end{proof}

\section{The exceptional graphs}\label{sec:identifications}

For ease of reference, the four exceptional graphs are collected in
Table~\ref{tab:exceptional}.

\begin{table}[ht]
\centering
\begin{tabular}{ccccc}
\hline
Quadric & Relations & $v$ & $k$ & $\{a,b\}$\\
\hline
$Q^+(3,9)$ & $R_2$         & $360$ & $135$ & $(45,54)$\\
$Q^-(3,9)$ & $R_2$         & $369$ & $108$ & $(27,36)$\\
$Q^-(3,5)$ & $R_1\cup R_2$ & $65$  & $34$  & $(15,18)$\\
$Q^+(3,7)$ & $R_1\cup R_3$ & $168$ & $111$ & $(70,75)$\\
\hline
\end{tabular}
\caption{The four exceptional strictly Deza graphs.}
\label{tab:exceptional}
\end{table}

\begin{prop}
The Deza children of the four exceptional graphs are as follows:
\[
\begin{array}{c|c|c}
\Gamma & a\text{-child} & b\text{-child}\\
\hline
\Gamma_{\{2\}}(Q^+(3,9))
    & R_3 & R_1\cup R_2\\
\Gamma_{\{2\}}(Q^-(3,9))
    & R_1\cup R_2 & R_3\\
\Gamma_{\{1,2\}}(Q^-(3,5))
    & R_2 & R_1\cup R_3\\
\Gamma_{\{1,3\}}(Q^+(3,7))
    & R_3 & R_1\cup R_2.
\end{array}
\]
\end{prop}

\begin{proof}
This follows immediately from the triples
\[
(54,54,45),\qquad
(27,27,36),\qquad
(18,15,18),\qquad
(75,75,70)
\]
of common-neighbour numbers computed in the proof of
Theorem~\ref{thm:classification}.
\end{proof}

\begin{prop}
The spectra of the four exceptional Deza graphs are
\[
\begin{aligned}
\Gamma_{\{2\}}(Q^+(3,9))
 &
 :\{135^1,15^{81},0^{128},(-9)^{150}\},\\
\Gamma_{\{2\}}(Q^-(3,9))
 &:
 \{108^1,9^{123},0^{164},(-15)^{81}\},\\
\Gamma_{\{1,2\}}(Q^-(3,5))
 &:
 \{34^1,4^{26},(-1)^{13},(-5)^{25}\},\\
\Gamma_{\{1,3\}}(Q^+(3,7))
 &:
 \{111^1,6^{64},(-1)^{54},(-9)^{49}\}.
\end{aligned}
\]
\end{prop}

\begin{proof}
The adjacency matrices $A_1,A_2,A_3$ belong to the Bose-Mesner algebra of the anisotropic association scheme and are therefore simultaneously diagonalizable. The result follows by specializing the
eigenmatrix of the primitive $3$-class subscheme given in \cite[Theorem~4.19 and Remark~4.20]{AdriaensenDeBoeck2025} and summing the eigenvalues corresponding to the relations occurring in each graph.
\end{proof}

\begin{rem}
For completeness, the spectra of the two children can also be read off from the same eigenmatrix. They are
\[
\begin{array}{c|c|c}
\Gamma & \operatorname{Spec}(\Gamma_a)
       & \operatorname{Spec}(\Gamma_b)\\
\hline
(360,135,54,45)
&
\{144^1,9^{128},0^{150},(-16)^{81}\}
&
\{215^1,15^{81},(-1)^{150},(-10)^{128}\}
\\[1mm]
(369,108,36,27)
&
\{188^1,8^{164},(-1)^{123},(-17)^{81}\}
&
\{180^1,16^{81},0^{123},(-9)^{164}\}
\\[1mm]
(65,34,18,15)
&
\{10^1,5^{13},0^{26},(-3)^{25}\}
&
\{54^1,2^{25},(-1)^{26},(-6)^{13}\}
\\[1mm]
(168,111,75,70)
&
\{63^1,7^{54},0^{64},(-9)^{49}\}
&
\{104^1,8^{49},(-1)^{64},(-8)^{54}\}.
\end{array}
\]
In particular, none of these children is strongly regular.
\end{rem}


\subsection{The elliptic graph on $65$ vertices and the Doro-Hall scheme}

Consider $q=5$ and $\varepsilon=-$. The relation graph $R_2$ has valency $10$. From the intersection numbers one obtains
\[
 p_{22}^1=0,\qquad p_{22}^2=3,\qquad p_{22}^3=2,
\]
and the three relation valencies are $(k_1,k_2,k_3)=(24,10,30)$. It follows that the distance partition of the graph $R_2$ with respect to a vertex is
\[
 1,\quad 10,\quad 30,\quad 24,
\]
corresponding respectively to $R_0,R_2,R_3,R_1$.  A direct use of the remaining intersection numbers gives the intersection array $\{10,6,4;1,2,5\}$ of such distance-regular graph. Hall proved that the graph with this intersection array is unique \cite{Hall1980}; it is the graph usually called the Doro-Hall (or Hall) graph, see also \cite[Section~12.2]{BrouwerCohenNeumaier1989}.
Consequently, $R_2\cong \Gamma_{\rm DH}$. Since the distance relations of the Doro-Hall graph are $R_2,R_3,R_1$, in this order, the relation $R_3$ is precisely its distance-$2$ graph. Hence
\[
\Gamma_{\{1,2\}}(Q^-(3,5))
=R_1\cup R_2
\cong \overline{\Gamma_{\rm DH}^{(2)}}.
\]

There is also a useful description of the Deza children. Since $(c_1,c_2,c_3)=(18,15,18)$, the $15$-child is $R_2$, and hence isomorphic to the Doro-Hall graph, whereas the $18$-child is$R_1\cup R_3=\overline{R_2}$. Thus the two children are the Doro-Hall graph and its complement.

\subsection{Hyperbolic cases as normal Cayley graphs}
Recall that the hyperbolic quadric $Q^+(3,q)$ is the Segre variety
\[
\mathcal S_{1,1}\cong PG(1,q)\times PG(1,q).
\]
Via the standard model of the Miquelian Minkowski plane, its non-degenerate plane sections, and hence the corresponding anisotropic points under polarity, may be identified with projectivities of
$PG(1,q)$. Consequently, the association scheme can be interpreted on $PGL(2,q)$: the relation containing two elements $f,g$ is determined
by their $PSL(2,q)$-cosets and by the number of fixed points of $fg^{-1}$ on $PG(1,q)$; see \cite[Remark~4.20]{AdriaensenDeBoeck2025}.  
Under this identification
\[
 R_1:\text{ one fixed point},\qquad
 R_2:\text{ two fixed points},\qquad
 R_3:\text{ no fixed points}.
\]
In particular, every union of these relations is a normal Cayley graph.

For $q=7$, the group $PSL(2,7)$ has non-trivial conjugacy classes of sizes
\[
 21,\quad 56,\quad 42,\quad 24,\quad24
\]
with element orders $2,3,4,7,7$, respectively; see the standard class
data in \cite{Atlas1985}. 
\begin{prop}
The graph $\Gamma_{\{1,3\}}(Q^+(3,7))$ is isomorphic to $$Cay(PSL(2,7),PSL(2,7)\setminus(\{1\}\cup3A))$$.
\end{prop}

\begin{proof}
We use the standard ATLAS notation for conjugacy classes. The group
$PSL(2,7)$ has a unique conjugacy class $3A$ of elements of order $3$, of size $56$. These are precisely the elements having two fixed points on $PG(1,7)$. Since $k_2=56$, the relation $R_2$ is $Cay(PSL(2,7),3A)$. As $\Gamma_{\{1,3\}}$ is the union of the other two non-trivial relations, the result follows by taking complements.
\end{proof}

For $q=9$ we use $PSL(2,9)\cong A_6$. 
\begin{prop}
The graph $\Gamma_{\{2\}}(Q^+(3,9))$ is the Cayley graph $Cay(A_6,2A\cup4A)$.
\end{prop}

\begin{proof}
Using $PSL(2,9)\cong A_6$, the three non-trivial relations have valencies
\[
(k_1,k_2,k_3)=(80,135,144).
\]
The non-trivial conjugacy classes of $A_6$ have sizes
\[
45,\ 40,\ 40,\ 90,\ 72,\ 72
\]
for the classes $2A,3A,3B,4A,5A,5B$, respectively. The elements with two fixed points on $PG(1,9)$ are precisely those in $2A\cup4A$.
Since
\[
|2A|+|4A|=45+90=135=k_2,
\]
the relation $R_2$ corresponds to $2A\cup4A$, proving the claim.
\end{proof}

\subsection{The elliptic graph on $369$ vertices}

For $\varepsilon=-$, the anisotropic association scheme admits the standard interpretation in terms of Baer sublines of $PG(1,q^2)$; see \cite[Remark~4.20]{AdriaensenDeBoeck2025}.

\begin{prop}
The graph $\Gamma_{\{2\}}(Q^-(3,9))$ is isomorphic to the disjointness graph on either of the two $PSL(2,81)$-orbits on the Baer sublines of $PG(1,81)$.
\end{prop}

\begin{proof}
The Baer sublines of $PG(1,q^2)$ split into two orbits under $PSL(2,q^2)$, corresponding to the two quadratic types of anisotropic points. Within either orbit, two Baer sublines belong to $R_1$, $R_2$, or $R_3$ according as their intersection has size $1$, $0$, or $2$, respectively. Hence $R_2$ is precisely the disjointness relation.
For $q=9$ each orbit has $\frac{9(9^2+1)}2=369$ elements.
\end{proof}

\section{Higher-dimensional quadrics}

The dimension-three phenomenon appears to be special.  For odd
dimensional quadrics $Q^\varepsilon(n,q)$, the intersection
numbers associated with the tangent relation have the form
\[
 p_{11}^{1}=2(q^{n-2}-1),
\]
and
\[
 p_{11}^{2}
 =2q^{(n-3)/2}
   \left(q^{(n-1)/2}-\varepsilon\right),
 \qquad
 p_{11}^{3}
 =2q^{(n-3)/2}
   \left(q^{(n-1)/2}+\varepsilon\right),
\]
with the appropriate labeling conventions; see
\cite{AdriaensenDeBoeck2025}. 
\begin{prop}
Let $Q>3$ and $n>3$ be odd. Then the tangency relation on one quadratic class of anisotropic points of $Q^\varepsilon(n,q)$ does not define a Deza graph.
\end{prop}

\begin{proof}
The three intersection numbers for the tangency relation are \[
p_{11}^1=2(q^{n-2}-1),
\]
\[
p_{11}^2=
2q^{(n-3)/2}
\left(q^{(n-1)/2}-\varepsilon\right),
\]
and
\[
p_{11}^3=
2q^{(n-3)/2}
\left(q^{(n-1)/2}+\varepsilon\right).
\]
Their pairwise differences are
\[
p_{11}^1-p_{11}^2
=
2\left(\varepsilon q^{(n-3)/2}-1\right),
\]
\[
p_{11}^1-p_{11}^3
=
-2\left(\varepsilon q^{(n-3)/2}+1\right),
\]
and
\[
p_{11}^2-p_{11}^3
=
-4\varepsilon q^{(n-3)/2}.
\]
Since $n>3$ and $q>1$, none of these quantities vanishes. Hence the three numbers are pairwise distinct.
\end{proof}

\begin{cor}
Among odd-dimensional hyperbolic and elliptic quadrics, the Deza property of the tangency relation occurs only in projective dimension 3.
\end{cor}

\begin{prob}
Classify the Deza graphs obtained as unions of relations in the
association schemes on one quadratic class of anisotropic points of
$Q^\varepsilon(n,q)$ for odd $n>3$.
\end{prob}

\section*{Acknowledgements}
The research was supported by the Italian National Group for Algebraic and Geometric Structures and their Applications (GNSAGA - INdAM and by the INdAM - GNSAGA Project \emph{Automorfismi di Strutture Geometriche Finite e codici lineari associati}, number E53C25002010001.

\section*{Declaration of the use of AI}
During the preparation of this manuscript, the author used ChatGPT (OpenAI) to assist with language editing and the presentation of some mathematical material. In particular, while investigating the link between anisotropic association scheme on $Q^+(3,q)$, and cosets of $PSL(2,q)$. and The generated suggestions were reviewed and, where appropriate, modified by the author. All mathematical results and references were independently checked. The author assumes full responsibility for the final content.

\end{document}